\documentclass[10pt,oneside,a4paper]{amsart}
\usepackage{latexsym}
\usepackage{amsfonts,amsmath,amssymb,indentfirst}
\usepackage[english]{babel}
\usepackage[all]{xy}
\usepackage[pdftex]{hyperref}
\newcommand{\bdism}{\begin{displaymath}}
\newcommand{\edism}{\end{displaymath}}
\newcommand{\cc}{\mathbb{C}}
\newcommand{\rr}{\mathbb{R}}

\newcommand{\zz}{\mathbb{Z}}

\newcommand{\pp}{\mathbb{P}}
\newcommand{\oo}{\mathcal{O}}

 \DeclareMathOperator{\Chow}{Chow}
\DeclareMathOperator{\Vol}{Vol}

\newtheorem{theorem}{Theorem}[section]
\newtheorem{proposition}[theorem]{Proposition}
\newtheorem{corollary}[theorem]{Corollary}
\newtheorem{lemma}[theorem]{Lemma}
\newtheorem{remark}[theorem]{Remark}

\newtheorem{definition}[theorem]{Definition}
\newtheorem{question}[theorem]{Question}

\address{Department of Mathematics, Princeton University, Princeton NJ 08544-1000,
USA} \email{gdi@math.princeton.edu}

\address{Department of Mathematics, Duke University, Durham NC 27708-0320,
USA} \email{luca@math.duke.edu}

\author{Gabriele Di Cerbo and Luca F. Di Cerbo}

\title{Effective results for complex hyperbolic manifolds}
\begin{document}

\maketitle
\pagestyle{headings}
\begin{abstract}

The goal of this paper is to study the geometry of cusped complex
hyperbolic manifolds through their compactifications. We
characterize toroidal compactifications with non-nef canonical
divisor. We derive effective very ampleness results for toroidal
compactifications of finite volume complex hyperbolic manifolds. We
estimate the number of ends of such manifolds in terms of their
volume. We give effective bounds on the number of complex hyperbolic
manifolds with given upper bounds on the volume. Moreover, we give
two sided bounds on their Picard numbers in terms of the volume and
the number of cusps.

\end{abstract}

\tableofcontents

\section{Introduction}
\pagenumbering{arabic}

Let $\mathcal{H}^{n}$ be the complex hyperbolic
space of dimension $n\geq2$, that is, the K\"ahler space form with
holomorphic sectional curvature $-1$. Let $X^{o}$ be a complete
non-compact complex hyperbolic manifold of finite volume. Such
manifold is obtained as $X^{o}:=\mathcal{H}^{n}/\Gamma$ where
$\Gamma$ is torsion-free lattice of $\textrm{PU}(1,n)$ with
parabolic elements which is non-uniform. Then $X^{o}$ has finitely
many disjoint unbounded ends of finite volume, the \emph{cusps} of
$X^{o}$.

Baily-Borel \cite{Borel1} and Siu-Yau \cite{Siu} proved that we can
find a compactification $X^{*}$ of $X^{o}$ such that the complement
of $X^{o}$ in $X^{*}$ consists of only finitely many (singular)
points, called cusp points. Furthermore, under some mild assumptions
on the lattice $\Gamma$, there exists a smooth variety $X$, which is
a resolution of $X^{*}$, such that each exceptional divisor over a
cusp point is a smooth abelian variety. For the precise requirements
on $\Gamma$ we refer the interested reader to Section
\ref{preliminaries}. Let us denote by $D=\coprod_{i} D_{i}$ the union
of the exceptional divisors on $X$. We will refer to the pair
$(X,D)$ as a toroidal compactification of $X^{o}$. The dimension of
$X$ will always be denoted by $n$ and assumed to be greater or equal
than two. Such toroidal compactifications were first constructed by
Mumford et al. \cite{Ash} and by Mok \cite{Mok}, for more details
see again Section \ref{preliminaries}.

The goal of this work is to understand how the positivity properties
of the log-canonical divisor $K_{X}+D$ determine the geometry of
$X^{o}$ and $(X, D)$. The first technical result of this paper is
that $K_{X}+D$ is not only big and nef but it is a limit of ample
divisors of the form $K_{X}+\alpha D$ where $\alpha$ is a real
number less than one. Moreover, there exists a uniform bound on the
ampleness range independent of the dimension. Remarkably, the
existence of such a uniform bound is crucial for all the geometric
applications presented in this work.

\begin{theorem}\label{Mok}
Let $(X,D)$ be a toroidal compactification. Then $K_{X}+\alpha D$ is
ample for all $\alpha\in\left(\frac{1}{3}, 1\right)$.
\end{theorem}

The proof of Theorem \ref{Mok} is based on the techniques developed
by the authors in \cite{DiCerbo}. We list here some of the main
applications. We refer to Section \ref{maggiore} for details.

The first bi-product of the proof of Theorem \ref{Mok} is a theorem
characterizing toroidal compactifications with non-nef canonical
divisor. Note that toroidal compactifications of ball quotients with
non-nef canonical divisor do exist, see the paper by Hirzebruch
\cite{Hirzebruch}. The surfaces constructed in \cite{Hirzebruch} are
blow ups of a particular Abelian surface at certain configurations
of points. These are the only currently known examples and
unfortunately they are all in dimension two! We hope Theorem
\ref{structure1} below will help in the construction of higher
dimensional examples and shed some light on the problem of
determining how common or rare these examples are.

\begin{theorem}\label{structure1}
Let $(X,D)$ be a toroidal compactification such that $K_{X}$ is not
nef. Then $X$ is the blow-up of a smooth variety $Y$ along a smooth
subvariety $Z$ of codimension $2$.
\end{theorem}

Theorem \ref{structure1} follows from a classification result of
extremal rays of toroidal compactifications. In particular, these
varieties appear to be very simple from the minimal model point of
view. This is quite unusual in higher dimension.

Next, Theorem \ref{Mok} can successfully be applied to study the
birational geometry of the log-canonical divisor of the pair $(X,
D)$.

\begin{theorem}\label{Mok1}
Let $(X,D)$ be a toroidal compactification. Then for all $m\geq
(2n+2)^{3}$, the map associated to $|m(K_{X}+D)|$ is a morphism and
it defines an embedding of $X\backslash D$ into some projective
space. Furthermore, this morphism maps all the different components
of $D$ to distinct singular points.
\end{theorem}

Theorem \ref{Mok1} is an effective version of a recent result of
Mok, see Main Theorem in \cite{Mok}. Nevertheless, our proof relies
on completely different techniques than Mok's. Note that when $(X,
D)$ is a toroidal compactification associated to a neat arithmetic
lattice in $\textrm{PU}(1,n)$, Theorem \ref{Mok1} gives a very
concrete realization of the classical Baily-Borel compactification
$X^{*}$ as a projective algebraic variety.

Theorem \ref{Mok} can also be applied to study the geometry of the
smooth variety $X$ in the pair $(X, D)$.

\begin{theorem}\label{degree1}
Let $(X,D)$ be a toroidal compactification. Then $X$ can be realized
as a smooth subvariety of $\pp^{2n+1}$ such that its degree $d$
satisfies \bdism d\leq (2n+2)^{3n}(K_{X}+D)^{n}. \edism
\end{theorem}

A similar result is obtained by Hwang in \cite{Hwa2}. The main
advantage of our approach is that it provides bounds which are
linear in $(K_{X}+D)^{n}$. The bounds obtained by Hwang are
polynomial of degree $n+1$ in $(K_{X}+D)^{n}$.

Theorem \ref{degree1} has several applications. For example,
combining this result with classical Chow variety techniques, we can
derive explicit bounds for the number of complex hyperbolic
manifolds with given upper bounds on the volume. These results are
effective versions of the classical Wang's finiteness theorem, see
\cite{Wan}. For the numerical values of these bounds see Corollaries
\ref{chow} and \ref{pairs}.

As a final application of Theorem \ref{Mok}, we show how to bound
the number of cusps of a finite volume complex hyperbolic manifold
in terms of its volume.

\begin{theorem}\label{3}
Let $(X,D)$ be a toroidal compactification. Let $q$ be the number of
cusps of $X^{o}$. Then
\begin{align}\notag
q\leq \left(\frac{3}{2}\right)^{n}(K_{X}+D)^{n}.
\end{align}
\end{theorem}

The problem of bounding the number of ends in pinched negatively
curved finite volume manifolds is a long standing problem in
differential geometry. For complex hyperbolic manifolds, the
classical approach is via techniques coming from geometric topology,
see \cite{Parker} and the bibliography therein. More recently, Hwang
in \cite{Hwa1} was able to improve Parker's result by using the
Hirzebruch-Mumford proportionality principle \cite{Mumford}. The
bound we present here in Theorem \ref{3} is better than the one
obtained by Parker but worse than Hwang's for $n\geq 4$. We decided
to include this result here mainly for two reasons. First, it
appears to be the best currently known bound for surfaces and
threefolds. Moreover, the proof of Theorem \ref{3} is completely
elementary and it follows easily from Theorem \ref{Mok}. Second, the
reasoning given in the proof of Theorem \ref{3} is in principle
applicable in a much general setting, while both the approaches of
Parker and Hwang use in an essential way the special properties of
complex hyperbolic manifolds.

In Section \ref{counting}, we describe an alternative approach to
the problem of bounding the cusps, see Proposition \ref{matsusaka}.
This approach does not rely on Theorem \ref{Mok} but rather on an
old result of Matsusaka. The strategy of the proof can again be in
principle applied to a much more general setting than ball
quotients, e.g., compactifications of finite volume K\"ahler
manifolds with pinched variable negative curvature.

Finally, in Section \ref{Picard1} we show how the techniques
developed in this paper can be successfully applied to bound from
above the Picard numbers of toroidal compactifications in terms of
their volume. Moreover, we give a lower bound on the Picard number
in terms of the number of cusps. For the details see Theorems
\ref{picard} and \ref{above}.

\subsection{Preliminaries}\label{preliminaries}

The theory of compactifications of locally symmetric varieties has
been extensively studied, see for example \cite{Borel2}. For
technical reasons this theory is mainly developed for quotients of
symmetric spaces by arithmetic subgroups. In most cases this is not
a serious issue since the work of Margulis \cite{Margulis} implies
that lattices in any semi-simple Lie group of real rank bigger or
equal than two are arithmetic subgroups. Nevertheless, this theorem
does not cover the case of lattices in the complex hyperbolic space
$\mathcal{H}^{n}$ which are the main object of study in this paper.
Note that non-arithmetic lattices in $\textrm{PU}(1,n)$ were
constructed by Mostow and Deligne-Mostow; see \cite{Deligne} and the
extensive bibliography there.

It would then be desirable to develop a theory of compactifications
of finite volume complex hyperbolic manifolds which does not rely on
the arithmeticity of the defining torsion free lattices.
Fortunately, this problem was addressed by many mathematicians from
several different point of views. A compactification of finite
volume complex hyperbolic manifolds as a complex spaces with
isolated normal singularities was obtained by Siu and Yau in
\cite{Siu}. This compactification may be regarded as a
generalization of the Baily-Borel compactification defined for
arithmetic lattices in $\mathcal{H}^{n}$. A toroidal
compactification for finite volume complex hyperbolic manifolds was
described by Hummel and Schroeder in \cite{Schroeder}. In many cases
these compactifications provide explicit resolution of singularities
of the Siu-Yau compactifications. More recently Mok \cite{Mok} gave
a lucid and detailed description of these compactifications and
described many of their remarkable features.

The compactifications studied in this work are the ones described by
Hummel-Schroeder and Mok. Thus, we do not require that lattice
$\Gamma$ to be arithmetic. On the other hand, we require that all
the parabolic isometries of $\Gamma$ are \emph{unipotent}, in other
words we require that they act by translations on their invariant
horospheres. We impose this condition in order to obtain smooth
toroidal compactifications. We would like to point out that this
technicality is hidden in the construction described by Mok while it
is explicitly discussed in the work of Hummel-Scroeder. Let us
describe in more details this technical point. Recall that given a
non-uniform torsion-free lattice $\Gamma\leq \textrm{PU}(1,n)$, the
finite volume complex hyperbolic manifold $\mathcal{H}^{n}/\Gamma$
has finitely many cusps $C_{1}, ..., C_{m}$ which are in one to one
correspondence with the maximal parabolic subgroups of $\Gamma$, see
\cite{Eberlein} for more details. Given a cusp $C_{i}$, denote by
$\Gamma_{i}\leq \Gamma$ the associated maximal parabolic subgroup
and by $\textrm{HB}_{i}$ the horoball stabilized by $\Gamma_{i}$.
After choosing an Iwasawa decomposition \cite{Eberlein} for
$\textrm{PU}(1,n)$, we can identify $\partial\textrm{HB}$ with a
Heisenberg type Lie group $N_{i}$ diffeomorphic to
$\cc^{n-1}\times\rr$. Thus, the center $Z_{i}$ of $N_{i}$ is
$Z_{i}=[N_{i}, N_{i}]$ and it is isomorphic to $\rr$. Furthermore,
the simply connected Lie group $N_{i}$ comes equipped with a natural
left invariant metric and then we can consider $\Gamma_{i}$ as a
lattice in $\textrm{Iso}(N_{i})$. The isometry group of $N$ is
isomorphic to the semi-direct product
$\textrm{Iso}(N_{i})=N_{i}\rtimes U(n-1)$. We then say that
$\phi\in\Gamma_{i}$ is unipotent if it is a translation in
$\textrm{Iso}(N_{i})$. Now, the construction described by
Hummel-Scroeder and Mok produces a smooth compactification if the
quotients $\Gamma_{i}/\Gamma_{i}\cap Z_{i}$ are torsion free for all
$i$. This is the case if all the parabolic isometries of $\Gamma$
are unipotent. In the classical arithmetic case this requirement is
usually satisfied by requiring the lattice $\Gamma$ to be neat, see
\cite{Borel1} and \cite{Borel2} for more details. Let us note that
given any non-uniform lattice $\Gamma\in\textrm{PU}(1,n)$, there
always exists a finite index subgroup whose all parabolic isometries
are unipotent, in the arithmetic case see \cite{Ash} while for the
general non-arithmetic case we refer to \cite{Hummel}.

\vspace{0.5cm}

\noindent\textbf{Acknowledgements}. The first named author would
like to express his gratitude to Professor J\'anos Koll\'ar for his
constant support and for valuable discussions on the results
contained in this work. Both authors are grateful to Cinzia
Casagrande for pointing out the reference \cite{Wis2} and for
generously sharing her knowledge. This reference was crucial in
obtaining a bound which is uniform in dimension in Theorem
\ref{Mok}. We thank the organizers of the conference ``Algebraic \&
Hyperbolic Geometry - New Connections'' where a preliminary version
of this work was presented and were part of the later improvements
were conceived.

\section{A gap theorem}

In this section we prove Theorem \ref{Mok} which is the main
technical result of the paper. The proof of Theorem \ref{Mok} is
based on the following result contained in \cite{DiCerbo}, see
Theorem 4.15.

\begin{theorem}\label{tian}
Let $X$ be a smooth projective variety and let $D$ be a reduced
effective divisor with simple normal crossing support such that
$K_{X}+D$ is big and nef. Then $K_{X}+\alpha D$ is ample for
$\alpha\in(\frac{n+1}{n+2}, 1)$ if and only if there are no
irreducible curves $C$ such that $(K_{X}+D)\cdot C=0$ and
$K_{X}\cdot C\leq0$.
\end{theorem}

The key ingredient in the proof of Theorem \ref{tian} is the Cone
Theorem. In particular the fact that $\alpha\in(\frac{n+1}{n+2}, 1)$
is obtained using the bound on the length of extremal rays given in
the Cone Theorem. Let us recall the definition of an extremal ray.
We refer to \cite{Kol2} and \cite{KM} for more details.

\begin{definition}
Let $N\subset \rr^{m}$ be a cone. A subcone $M\subset N$ is called
extremal if $u,v\in N$, $u+v\in M$ imply that $u,v\in M$. A
$1$-dimensional extremal subcone is called an extremal ray.
\end{definition}

In this paper we will consider only extremal rays $R$ of the cone of
effective $1$-cycles $\overline{NE}(X)$ such that $K_{X}\cdot Z<0$
for any effective $1$-cycle $Z$ in $R$. For such extremal rays we
define the length of $R$ as \bdism l(R):=\min\left\{-K_{X}\cdot C
\:|\: \text{$C$ is a rational curve with numerical class in $R$}
\right\}. \edism

The length of extremal rays in $\overline{NE}(X)$ has been
extensively studied. The Cone Theorem gives that $l(R)\leq n+1$. On
the other hand, there are many classification results regarding
varieties with extremal rays of maximal length. We will prove a
strong bound on the length of extremal rays of toroidal
compactifications.

\begin{theorem}\label{raggi}
Let $(X,D)$ be a toroidal compactification of dimension $n$. Then
$l(R)\leq 1$ for any extremal ray $R$.
\end{theorem}

We will prove Theorem \ref{raggi} in few steps. The first and key
result is the following.

\begin{lemma}\label{fiberdim}
Let $(X,D)$ be a toroidal compactification. Let $R$ be an extremal
ray of $X$ and let $F$ be an irreducible component of a non-trivial
fiber of the contraction of $R$. Then $\dim(F)\leq 1$.
\end{lemma}

\begin{proof}
Let $C$ be a curve in $F$. Since $C$ is contracted by the
contraction of $R$, the Contraction Theorem, see \cite{KM}, gives
that $C\in[R]$. Since $R$ is $K_{X}$-negative and $K_{X}+D$ is nef,
we have that $D\cdot C>0$. On the other hand, if $C\subseteq D_{i}$,
then $D\cdot C<0$ because the normal bundle of $D_{i}$ is
anti-ample. In particular, $\dim(D\cap F)=0$ which implies
$\dim(F)\leq 1$.
\end{proof}

The dimension of the fibers of the contraction controls the length
of the corresponding extremal ray. Recall the following result in
\cite{Wis2}.

\begin{theorem}[Wi\'sniewski]\label{Cinzia}
If $F$ is a nontrivial fiber of a contraction of $R$ then \bdism
\dim(F)\geq l(R)-1. \edism
\end{theorem}

Combining Theorem \ref{Cinzia} and Lemma \ref{fiberdim} we obtain
the following.

\begin{corollary}
Let $(X,D)$ be a toroidal compactification and let $R$ be an
extremal ray of $X$. Then $l(R)\leq 2$.
\end{corollary}

In order to prove Theorem \ref{raggi}, it suffices to prove that the
case $l(R)=2$ and $\dim(F)\leq 1$ does not occur. Such extremal
contractions have been classified by Wi\'sniewski \cite{Wis2}.

\begin{theorem}[Wi\'sniewski]\label{Wisn}
Let $X$ be a smooth variety. Let $\phi: X\rightarrow Y$ be the
contraction of an extremal ray $R$ of $X$ such that $\dim(F)\leq 1$.
Then $Y$ is smooth and either
\begin{enumerate}
\item $\phi: X\rightarrow Y$ is a conic bundle, or
\item $\phi: X\rightarrow Y$ is the blow-up of the variety $Y$ along a smooth subvariety $Z$ of codimension $2$.
\end{enumerate}
\end{theorem}

Our goal is to prove that toroidal compactifications do not admit a
conic bundle structure.  This concludes the proof of Theorem
\ref{raggi}.

\begin{proof}[Proof of Theorem \ref{raggi}]

By Lemma \ref{fiberdim} and Theorem \ref{Wisn} we need to show that
$\phi: X\rightarrow Y$ cannot be a conic bundle. We first deal with
the case $\dim(X)=2$. In this case $l(R)=2=\dim(X)$ and by Proposition 2.4.2
in \cite{Wis}, we obtain that the Picard number of $X$, say $\rho(X)$, is equal to two. Since $D$ is an elliptic curve, we have
that $\phi$ restricted to $D$ is a surjective morphism and in
particular $g(Y)\leq 1$. Since $\rho(X)=2$, the ruled surface $X$
must be minimal. If $Y$ is a rational curve then $X$ is a Hirzebruch
surface. On a Hirzebruch surface the only curve with negative self
intersection is the zero section, which is a rational curve, so it
cannot be $D$.

We can then assume that $Y$ is an elliptic curve. Since $X$ is minimal, we know that $X=\pp(\mathcal{E})$ where $\mathcal{E}$ is a normalized rank 2 vector bundle over $Y$. 
Following Proposition V.2.8 in \cite{Har}, we can normalize $\mathcal{E}$ so that the integer
$e=-deg(\mathcal{E})$ is an invariant on $X$. Recall that, under this normalization, the canonical section of $X$, say $C_{0}$, satisfies $C^{2}_{0}=-e$.
Since $Y$ is elliptic, by a theorem of Atiyah we know that $e\geq -1$, see
Theorem V.2.15 in \cite{Har}. If $e\leq 0$ then $-K_{X}$ is nef and
in particular by adjunction $D^{2}\geq 0$. We then have $e\geq 1$. By Proposition V.2.20 in \cite{Har}, any irreducible curve $Y\neq C_{0}$, $F$ must be numerically equivalent to
\begin{align}\notag
Y\equiv aC_{0}+bF,
\end{align}
where $F$ is a fiber of the ruling and $a>0$, $b\geq ae$. Thus
\begin{align}\notag
Y^{2}=-a^{2}e+2ab\geq a^{2}e\geq 0,
\end{align}
which then shows that the only curve with negative self-intersection is the zero section.
This implies that $D$ is the zero section. We can then find a
$\pp^{1}$ with just one puncture in $X\backslash D$. Since curvature
can only decrease along complex submanifolds of a negatively curved
K\"ahler manifold, we obtain a contradiction.

Now suppose that $\dim(X)\geq 3$ and $\phi$ is a conic bundle. We will 
work on the Baily-Borel compactification $X^{*}$. Recall that $X^{*}$ is 
obtained from $X$ contracting the components of the divisor $D$. In particular, 
the fibers of $\phi$ define a family of rational curves on $X^{*}$. Let $F$ be a fiber 
of $\phi$. First, assume that every fiber of $\phi$ is a smooth conic. Since 
$D\cdot F\geq 3$, the family of rational curves on $X^{*}$ has at 
least three fixed points. Recall that bend and break, see Lemma 1.9 in \cite{KM}, 
implies that if a family of rational curves moves with at least two fixed points then  
it must split, i.e. it contains a reducible curve.  Applying bend and break to our 
situation, we obtain that the family of rational curves on $X^{*}$ contains 
a reducible curve which corresponds to a reducible fiber of $\phi$. This contradicts our 
assumption on the smoothness of the fibers. 

In general, $\phi$ has reducible fibers. Since the discriminant locus of a conic bundle 
is a divisor on the base, the reducible fibers form a new family of rational curves on $X^{*}$ of positive dimension because $\dim(X)\geq 3$. The family formed by the irreducible components has at least two fixed points and applying again bend and break, we obtain a new splitting. In particular, we have a fiber with three components. This is a contradiction because every fiber of $\phi$ is isomorphic to a conic in $\pp^{2}$. 

\end{proof}

Theorem \ref{raggi} implies an interesting structure result on
toroidal compactifications with non-nef canonical divisor.

\begin{corollary}[Theorem \ref{structure1}]\label{structure}
Let $(X,D)$ be a toroidal compactification such that $K_{X}$ is not
nef. Then $X$ is the blow-up of a smooth variety $Y$ along a smooth
subvariety $Z$ of codimension $2$.
\end{corollary}


We can now prove Theorem \ref{Mok}.

\begin{proof}[Proof of Theorem \ref{Mok}]
First, we would like to show that $K_{X}+D$ is big and nef. Recall
that $\mathcal{H}^{n}$ is the negatively curved complex space form.
Thus, any manifold $X^{o}=\mathcal{H}^{n}/\Gamma$, with
$\Gamma$ non-uniform, is equipped with a standard negatively curved
K\"ahler-Einstein metric of finite volume which we denote by
$\hat{\omega}$. By Th\'eor\`eme 1.1 in \cite{Sibony}, such a metric
$\hat{\omega}$ on $X^{o}$ can be regarded as a closed
\emph{positive} current on $X$. Moreover, it in not difficult to see
that $\hat{\omega}$ is in the cohomology class of $K_{X}+D$. More
precisely, let $D_{i}$ be the irreducible components of $D$ and let
$s_{i}\in \mathcal{O}_{X}(D_{i})$ be the defining sections. By
appropriately choosing Hermitian metrics $\|\cdot\|_{i}$ on
$\mathcal{O}_{X}(D_{i})$, the volume form associated to
$\hat{\omega}$ is then of the form
\begin{align}\notag
\Psi=\frac{\Omega}{\prod_i\|s_{i}\|^{2}(-\log\|s_{i}\|^{2})^{n+1}}
\end{align}
for some globally defined volume form $\Omega$ on $X$. Thus, the
Poincar\'e-Lelong formula combined with the fact that $\hat{\omega}$
is K\"ahler-Einstein gives that $\hat{\omega}\in[K_{X}+D]$. Next,
let us observe that $\omega_{P}\geq\hat{\omega}$ for some K\"ahler
current $\omega_{P}$ on $X$ with Poincar\'e type singularities along
$D$. We then have that $\hat{\omega}$ is a closed positive current
with zero Lelong numbers on $X$. Moreover, since $\hat{\omega}$ is a
regular K\"ahler metric on $X^{o}$, it is strictly positive at any
point $p\in X\backslash D$. By using a regularization argument based
on results of Demailly \cite{Demailly1}, we conclude that $K_{X}+D$
is big and nef, see also Theorem 1.3 in \cite{Shiffman}. We can
actually conclude more. In fact, we must have
$(K_{X}+D)^{\dim(Z)}\cdot Z>0$ for any subvariety $Z$ not contained
in $D$.

Finally, we want to show that if $C$ is a curve such that
$(K_{X}+D)\cdot C=0$ then $K_{X}\cdot C> 0$. By the discussion above
we know that $C$ must be contained in $D$. Then
\begin{align}\notag
K_{X}\cdot C=-D\cdot C>0,
\end{align}
since the normal bundle of $D$ in $X$ is anti-ample, see Theorem 1
in \cite{Mok}. In particular Theorem \ref{tian} implies that
$K_{X}+\alpha D$ is ample for all values of $\alpha$ close enough to
one. Let us denote by $\{C_{i}\}$ the generators of the extremal
rays in $X$. Each $C_{i}$ is a smooth rational curve and $D\cdot
C_{i}\geq 3$ for all $i$. If otherwise there would exists a
$\pp^{1}$ with less than three punctures in $X\backslash D$ which is
clearly impossible. We then conclude that $(K_{X}+D)\cdot C_{i}\geq
2$ for all $i$. Following the proof of Theorem \ref{tian} in
\cite{DiCerbo} and using the bound given by Theorem \ref{raggi}, it
follows that the $K_{X}+\alpha D$ is ample for all
$\alpha\in\left(\frac{1}{3},1\right)$.
\end{proof}

\begin{remark}
Theorem \ref{Mok} is geometrically sharp as it relates the existence
of a ``gap'', in the ampleness range of $K_{X}+\alpha D$, only to
the number of times a smooth rational curve in the ambient space $X$
intersects the boundary $D$. Elementary hyperbolic geometry tells us
that this number must be greater or equal than three, and then
the bound in Theorem \ref{Mok}. Unfortunately, in all of the
examples constructed by Hirzebruch in \cite{Hirzebruch}, the
rational curves intersect the boundary in four points. Thus, Theorem
\ref{Mok} is currently not numerically sharp. It would be extremely
interesting to construct examples of toroidal compactifications with
non-nef canonical divisor, having a rational curve that intersects
the boundary in just three points.
\end{remark}

It seems a difficult problem to understand which varieties arise as
toroidal compactifications of hyperbolic manifolds. The above
theorem gives a first step toward a possible solution of the
problem. If $n\geq 3$, it follows also from Lefschetz hyperplane
theorem.

\begin{corollary}
There are no toroidal compactification $(X,D)$ with $X$ a smooth
Fano variety.
\end{corollary}

\begin{proof}
Suppose $(X,D)$ is a toroidal compactification with $-K_{X}$ ample.
Because of Theorem \ref{Mok}, for all $\alpha$ close to one, we know
$K_{X}+\alpha D$ is ample. By Corollary 4.18 in \cite{DiCerbo}, it
follows that $K_{X}+D$ is strictly nef. On the other hand, we must
have $(K_{X}+D)_{|_{D}}=\oo_{D}$. This is a contradiction.
\end{proof}

\begin{question}
It is interesting to ask whether there exists a smooth toroidal
compactification of a ball quotient with negative Kodaira dimension.
\end{question}

\section{Applications}\label{maggiore}

In this section, we give the proofs of Theorems \ref{Mok1},
\ref{degree1} and \ref{3} stated in the Introduction.

\subsection{Effective birationality}

Let us start by studying the birational properties of the divisor
$K_{X}+D$. First, we prove that the map associated to $|m(K_{X}+D)|$
maps the components of $D$ to distinct points for any $m\geq 2$.

\begin{proposition}\label{exact}
Let $(X,D)$ be a toroidal compactification. Then for any $i$ there
exists a section $\sigma_{i}$ of $H^{0}(X,\oo_{X}(2(K_{X}+D)))$ such
that $\sigma_{i}|_{D_{i}}\neq 0$ and $\sigma_{i}|_{D_{j}}=0$ for all
$j\neq i$.
\end{proposition}

\begin{proof}
Write $D=\sum_{i=1}^{q}D_{i}$ and recall that each component is an
abelian variety and they are all disjoint. Consider the following
exact sequence \bdism 0\rightarrow \oo_{X}(2K_{X}+D)\stackrel{\cdot
D}{\rightarrow} \oo_{X}(2(K_{X}+D))\rightarrow \oo_{D}\rightarrow 0.
\edism By Kawamata-Viehweg's vanishing we have that
$H^{1}(X,\oo_{X}(2K_{X}+D))=0$. Thus, taking the long exact sequence
in cohomology, we get the following surjective map \bdism
H^{0}(X,\oo_{X}(2(K_{X}+D)))\rightarrow \bigoplus_{i=1}^{q}
H^{0}(D_{i},\oo_{D_{i}}). \edism
\end{proof}

The next step is to understand what happens outside the boundary
divisor. Let us start by deriving a lower bound on the top
self-intersection of $K_{X}+D$.

\begin{lemma}\label{gromov}
Let $(X,D)$ be a toroidal compactification. Then \bdism
(K_{X}+D)^{n}\geq (n+1)^{n-1}. \edism
\end{lemma}

\begin{proof}

Since the K\"ahler-Einstein current $\hat{\omega}$ can be multiplied
by itself, the top self-intersection of $L:=K_{X}+D$ can be
expressed in terms of the Riemannian volume of $X^{o}$. More
precisely, by normalizing the holomorphic sectional curvature to be
$-1$ we have
\begin{align}\notag
\Vol(X)=\frac{(4 \pi)^{n}}{n!(n+1)^{n}} L^{n}.
\end{align}
On the other hand, Gromov-Harder's generalization of Gauss-Bonnet
\cite{Gro1} implies that
\begin{align}\notag
\Vol(X)\geq \frac{(4\pi)^{n}}{(n+1)!}.
\end{align}
Combining the two formulas above we get the result.

\end{proof}

These considerations and a theorem of Koll\'ar in \cite{Kol} imply
the following.

\begin{corollary}
Let $(X,D)$ be a toroidal compactification. Then $m(K_{X}+D)$ is
base point free for any $m\geq \binom{n}{2}+1$.
\end{corollary}

\begin{proof}

By Theorem \ref{Mok}, for any subvariety $Z$ not contained in $D$ we
have $(K_{X}+D)^{\dim(Z)}\cdot Z>0$. Because of Theorem 5.8 in
\cite{Kol}, we know that $m(K_{X}+D)$ is free at all points outside
$D$ for any $m\geq \binom{n}{2}+1$. Moreover, by Lemma \ref{exact}
we know already that $2(K_{X}+D)$ is free on the divisor $D$.

\end{proof}

Similarly, we can study separation of points.

\begin{corollary}\label{separation}
Let $(X,D)$ be a toroidal compactification. Then $m(K_{X}+D)$
separates any two points in $X\backslash D$ for any $m\geq
\binom{n}{2}+2$.
\end{corollary}

\begin{proof}
Combine Theorem 5.9 in \cite{Kol} and Lemma \ref{gromov}.
\end{proof}

Since $X\backslash D$ is open in $X$, Corollary \ref{separation}
implies that $|m(K_{X}+D)|$ defines a birational map for any $m\geq
\binom{n}{2}+2$. These simple results can be already used to
slightly improve all the bounds in \cite{Hwa2}. Nevertheless, the
approach we follow here is quite different. Note that Corollary
\ref{separation} deals only with separation of points. Thus, if we
want to prove that $m(K_{X}+D)$ defines an embedding outside $D$ we
need something more. Again the key for us is Theorem \ref{Mok},
while the approach described in \cite{Hwa2} relies on Seshadri
constants type arguments. Let us start by recalling the following
definition.

\begin{definition}
Let $X$ be a smooth projective variety and let $D$ be an effective
divisor. We say that a divisor $L$ is very ample modulo $D$ if the
map $\phi_{L}:X \dashrightarrow \pp(H^{0}(X,\oo_{X}(L)))$ defines an
embedding of $X\backslash D$.
\end{definition}

We can now derive an effective result on very ampleness modulo $D$
of $K_{X}+D$. In order to keep the final statement simple, some of
the constants used in the proof are not optimal.

\begin{theorem}\label{veryample}
Let $(X,D)$ be a toroidal compactification of dimension $n$. Then
$m(K_{X}+D)$ is very ample modulo $D$ for any $m\geq 2(n+1)^{3}$.
\end{theorem}

\begin{proof}
Theorem \ref{Mok} implies that $K_{X}+\frac{1}{2}D$ is ample. Then
$H:=2K_{X}+D$ is an ample integral divisor. By a theorem of
Angehrn-Siu we have that $K_{X}+mH$ is ample and base point free for
any $m> \binom{n+1}{2}$, see \cite{Ang} or \cite{Kol} and
\cite{Laz2} for an algebraic proof. In particular
$B:=K_{X}+(n^{2}+1)H$ is ample and base point free. A corollary of
Castelnuovo-Mumford regularity, see Example 1.8.23 in \cite{Laz1},
gives that $K_{X}+(n+2)B$ is very ample. All together we have that
the divisor \bdism
\left(2(n^{2}+1)(n+2)+n+3\right)K_{X}+(n^{2}+1)(n+2)D \edism is very
ample. Let $M:=\left(2(n^{2}+1)(n+2)+n+3\right)$. Adding the right
positive multiple of $D$ we get the following injective map \bdism
H^{0}(X,\oo_{X}(K_{X}+(n+2)B))\hookrightarrow
H^{0}(X,\oo_{X}(M(K_{X}+D))). \edism This implies that $m(K_{X}+D)$
is very ample modulo $D$ for any $m\geq M$. Since $2(n+1)^{3}\geq M$
we get the statement of the theorem.
\end{proof}

Theorem \ref{Mok1} stated in the introduction is now just a
combination of Theorem \ref{veryample} and Proposition \ref{exact}.
\vspace{0.5cm}

Let us conclude this section by discussing the connections between
Fujita's conjecture and Theorem \ref{veryample}. First, recall that
if $\dim(X)\leq 4$ then the base point free part of Fujita's
conjecture is known to be true thanks to the work of Kawamata
\cite{Kawamata}. We can then improve the bound in Theorem
\ref{veryample} for low dimensional varieties.

\begin{corollary}
Let $(X,D)$ be a toroidal compactification of dimension $n\leq 4$.
Then $m(K_{X}+D)$ is very ample modulo $D$ for any $m\geq
2(n+2)^{2}$.
\end{corollary}

\begin{proof}
It follows along the same lines of Theorem \ref{veryample}. Instead
of using Angehrn-Siu's theorem, we use Kawamata's result on Fujita's
conjecture. The rest remains unchanged.
\end{proof}

It is interesting to note that using a theorem of G. Heier in
\cite{Heier}, we can get a bound of order $n^{7/3}$ in Theorem
\ref{veryample}. Of course an even better bound would be obtained
using the very ampleness part of Fujita's conjecture. More
precisely, assuming the conjecture of Fujita to be true, the proof
of Theorem \ref{veryample} easily gives that $m(K_{X}+D)$ is very
ample modulo $D$ for any $m\geq 2(n+2)$.

\subsection{Bounds on the number of cusps}\label{counting}

In this section, we show how to explicitly estimate the number of
cusps of $X^{o}$ in terms of the intersection number
$(K_{X}+D)^{n}$. We present two different approaches. The first one
is a direct corollary of Theorem \ref{Mok}, while the second bound
follows easily from a result of Matsusaka. Moreover, it gives the
best currently known bound for threefolds.

Recall that $X^{o}$ has finitely many ends which correspond to the
cusp points of $X^{*}$. Each cusp point gives rise to one component
of the boundary divisor $D$ in $X$. So bounding the number of cusps
is equivalent to bounding the number of components of $D$.

We can now prove Theorem \ref{3} stated in the Introduction.

\begin{theorem}[Theorem \ref{3}]
Let $(X,D)$ be a toroidal compactification. Let $q$ be the number of
components of $D$. Then \bdism q\leq
\left(\frac{3}{2}\right)^{n}(K_{X}+D)^{n}. \edism
\end{theorem}

\begin{proof}
Once we have Theorem \ref{Mok}, the above result follows from a
quite straightforward computation. Let $L:=K_{X}+D$. Because of
Theorem \ref{Mok}, we know that $3L-2D$ is nef and $2L-D$ is ample.
Then \bdism q\leq D\cdot (2L-D)^{n-1}=(-1)^{n-1}D^{n}. \edism Since
$3L-2D$ is nef we get that \bdism
(3L-2D)^{n}=3^{n}L^{n}+(-1)^{n}2^{n}D^{n}\geq 0. \edism Combining
the two inequalities we get the result.
\end{proof}

Note that we can also a bound on the top self-intersection of $D$ in
terms of $(K_{X}+D)^{n}$ only. Let us summarize this fact in the
form of a corollary as it will be used in Section \ref{chowbound}.

\begin{corollary}\label{top}
Let $(X,D)$ be a toroidal compactification. Then \bdism
(-1)^{n-1}D^{n}\leq \left(\frac{3}{2}\right)^{n}(K_{X}+D)^{n}.
\edism
\end{corollary}

We now present an estimate on the number of cusps which does not
rely on Theorem \ref{Mok}. The key point is the following corollary
of Proposition \ref{exact}.

\begin{corollary}\label{cusp2}
Let $(X,D)$ be a toroidal compactification. Let $q$ be the number of
components of $D$. Then \bdism h^{0}(X,\oo_{X}(2(K_{X}+D)))\geq q.
\edism
\end{corollary}

In particular, in order to bound the number of cusps, it is enough
to bound $h^{0}(X,\oo_{X}(2(K_{X}+D))$. This can be done using a
result of Matsusaka. For an outline of the proof we refer to
VI.2.15.8.7 page 301 in \cite{Kol2}.

\begin{theorem}[Matsusaka]\label{matsu}
Let $H$ be a big and nef divisor on a smooth variety $X$ of
dimension $n$. Then for any $m\geq 1$ we have \bdism
h^{0}(X,\oo_{X}(mH))\leq m^{n}H^{n}+n. \edism
\end{theorem}

Combining these two results we get the following.

\begin{proposition}\label{matsusaka}
Let $(X,D)$ be a toroidal compactification. Let $q$ be the number of
components of $D$. Then \bdism q\leq (2^{n}+1)(K_{X}+D)^{n}. \edism
\end{proposition}

\begin{proof}
By Lemma \ref{gromov}, we have that $(K_{X}+D)^{n}\geq n$.
\end{proof}

It is interesting to note how the bounds given in Theorem \ref{3}
and Proposition \ref{matsusaka} rely on few basic geometric
properties of the pair $(X, D)$. Roughly speaking, they depend on
the fact that $L=K_{X}+D$ is big, nef and such that
$L_{|_{D}}=\oo_{D}$. None of the other special features of toroidal
compactifications of ball quotients is used. Therefore, these
techniques can be successfully applied to study complete finite
volume K\"ahler manifolds with pinched negative sectional curvature
which are not locally symmetric. These results will appear
elsewhere.

Finally, despite their simple nature, the bounds given in Theorem
\ref{3} and Proposition \ref{matsusaka} are the best currently known
bounds for surfaces and threefolds, compare with \cite{Parker} and
\cite{Hwa1}.

\subsection{Bounds on the number of varieties}\label{chowbound}

In this section we study the problem of finding effective bounds on
the number of complex hyperbolic manifolds with bounded volume. The
main issue is to find effective embedding results for toroirodal
compactifications. This was successfully done in Theorem
\ref{veryample} and we investigate its consequences here. The key
result of this section is Theorem \ref{degree1} stated in the
Introduction, which we restate for the convenience of the reader.

\begin{theorem}[Theorem \ref{degree1}]\label{embed}
Let $(X,D)$ be a toroidal compactification. Then $X$ can be realized
as a smooth subvariety of $\pp^{2n+1}$ such that its degree $d$
satisfies \bdism d\leq (2n+2)^{3n}(K_{X}+D)^{n}. \edism
\end{theorem}

\begin{proof}
Let $L:=K_{X}+D$. By Theorem \ref{veryample}, we know that
$2(n+1)^{3}L-kD$ is very ample, for some explicit positive integer
$k$. Then $2(n+1)^{3}L-kD$ defines an embedding of $X$ into some
projective space with degree $d\leq \left(2(n+1)\right)^{3n}L^{n}$.
By general projection, we obtain that $X$ sits inside $\pp^{2n+1}$.
\end{proof}

As pointed out in the Introduction, Hwang proved a similar result in
\cite{Hwa2}. While our bound is linear in $(K_{X}+D)^{n}$, he
derived a polynomial bound in $(K_{X}+D)^{n}$ of degree $n+1$. Once
we have Theorem \ref{embed}, counting the number of complex
hyperbolic manifolds is reduced to standard techniques on the
complexity of Chow varieties. Let us briefly explain this point.
Associated to a hyperbolic manifold, we have a toroidal
compactification $(X,D)$. Therefore, it is enough to count such
pairs. Fujiki in \cite{Fujiki} proved that a toroidal
compactification $(X,D)$ is infinitesimally rigid under deformations
of the pair. This implies that the number of toroidal
compactifications is bounded by the number of components of a
suitable Chow variety, see Corollary \ref{pairs}. Moreover, Hwang
pointed out that Fujiki's theorem implies something more. If $n\geq
3$, $X$ itself is rigid under deformations and the same method
applies, see Proposition 4.2 in \cite{Hwa2}.

We start by counting the varieties which arise as toroidal
compactifications of complex hyperbolic manifolds with bounded
volume. In particular, in the next statement we forget about the
extra structure coming from the boundary divisor.

Fix two positive integers  $d$ and $m>n$. We denote by
$\Chow_{m}(n,d)$ the Chow variety of $n$-dimensional irreducible
smooth subvarieties of $\pp^{m}$ of degree $d$.

\begin{corollary}\label{chow}
Fix two positive integers $n\geq 3$ and $V$. Then the number of
varieties $X$ arising as toroidal compactifications $(X,D)$ with
$\dim(X)=n$ and $(K_{X}+D)^{n}\leq V$ is bounded by \bdism
\sum_{d=1}^{d_{0}}\binom{(2n+2)d}{2n+1}^{(2n+2)d\binom{d+n-1}{n}+(2n+2)\binom{d+n-1}{n-1}},
\edism where $d_{0}=\left(2n+2\right)^{3n}V$.
\end{corollary}

\begin{proof}
Let $(X,D)$ be a toroidal compactification as in the statement. By
Theorem \ref{embed}, $X$ can be embedded in $\pp^{2n+1}$ as a smooth
subvarieties of degree $d\leq d_{0}$. Proposition 4.2 in \cite{Hwa2}
implies that $X$ is rigid under deformation. In particular, the
number of such toroidal compactifications is bounded by the number
of components of the Chow variety $\Chow_{2n+1}(n,d_{0})$. A
straightforward application of I.3.28.9 in \cite{Kol2} gives the
result.
\end{proof}

\begin{remark}
Independently of the value of $d_{0}$, the bound on the number of
components of the Chow variety is of the form \bdism
\left(a_{n}d_{0}\right)^{\left(b_{n}d_{0}\right)^{n+1}}, \edism for
some function $a_{n}$ and $b_{n}$ which depend only on $n$.
Therefore, the bound in Corollary \ref{chow}, expressed just in
terms of the volume, is of the form \bdism V^{V^{n+1}}. \edism
\end{remark}

We can now give effective estimates on the number of toroidal
compactifications $(X,D)$ with bounded $(K_{X}+D)^{n}$, counted as
pairs. This corresponds exactly to bounding the number of complex
hyperbolic manifolds with bounded volume. Unfortunately the bounds
are slightly worst than the ones in Corollary \ref{chow}.

We need to study a different Chow variety which takes into account
the boundary divisor. Given positive integers
$q,m,n,n_{1},\dots,n_{q},d,d_{1},\dots,d_{q}$ we define the Chow
variety $\Chow_{m}(n,d;n_{1},\dots,n_{q};d_{1},\dots,d_{q})$ to be
the closed subvariety of \bdism \Chow_{m}(n,d)\times
\Chow_{m}(n_{1},d_{1})\times \cdots \times \Chow_{m}(n_{q},d_{q}),
\edism parametrizing $(q+1)$-tuples $(X,D_{1},\dots,D_{q})$ where
$X$ is a smooth subvariety of dimension $n$ and degree $d$ in
$\pp^{m}$ and each $D_{i}$ is a smooth subvariety of dimension
$n_{i}$ and degree $d_{i}$ contained in $X$.

In \cite{Hwa2}, Hwang generalizes the effective bounds on the number
of components in \cite{Kol2} to the above defined Chow varieties. We
will use his bound to derive an effective estimate on the number of
toroidal compactifications.

The first thing to do is to bound the degree of each component of
$D$ when embedded into some fixed projective space.

\begin{lemma}\label{degree2}
Let $(X,D)$ be a toroidal compactification. Let $d_{i}$ be the
degree of $D_{i}$ under the embedding given by Theorem \ref{embed}.
Then for any $i$ \bdism d_{i}\leq
\left(\frac{3}{2}\right)^{n}(n+1)^{3n-3}(K_{X}+D)^{n}. \edism
\end{lemma}

\begin{proof}
Let $A$ be the very ample divisor obtained in the proof of Theorem
\ref{veryample}. Recall that this divisor is of the form
\begin{align}\notag
A=M(K_{X}+D)-(n^{3}+2n+5)D
\end{align}
for some positive integer $M$. Since $(n^{3}+2n+5)<(n+1)^{3}$, we obtain that \bdism d_{i}=D_{i}\cdot
A^{n-1}\leq (-1)^{n-1}(n+1)^{3n-3}D^{n}\leq
\left(\frac{3}{2}\right)^{n}(n+1)^{3n-3}(K_{X}+D)^{n}, \edism where
the last inequality is obtained thanks to Lemma \ref{top}.
\end{proof}

Finally, we can prove an effective version of Wang's finiteness
theorem.

\begin{corollary}\label{pairs}
Fix two positive integers $n$ and $V$. Then the number of toroidal
compactifications with $\dim(X)=n$ and $(K_{X}+D)^{n}\leq V$ is less
than \bdism
\sum_{q=1}^{q_{0}}d_{0}^{q+1}\binom{(2n+2)d_{0}}{2n+1}^{(2n+2)(q+1)\binom{2n+1+d_{0}}{2n+1}},
\edism where $d_{0}=(2n+2)^{3n}V$ and $q_{0}=(3/2)^{n}V$.
\end{corollary}

\begin{proof}
Let $d$ be the degree of $X$ in $\pp^{2n+1}$ and let $d_{i}$ be the
degree of $D_{i}$ for any $1\leq i\leq q$. By definition, there
exists a point in
$\Chow_{2n+1}(n,d;n-1,\dots,n-1;d_{1},\dots,d_{q})$ parametrizing
$(X,D)$. Theorem 4.1 in \cite{Fujiki} implies that $(X,D)$ is
infinitesimally rigid under deformations of the pair. Therefore, the
number of toroidal compactifications with fixed
$(d,d_{1},\dots,d_{q})$ is bounded by the number of components of
the corresponding Chow variety. By Theorem \ref{embed}, we can embed
$X$ as a smooth subvariety of $\pp^{2n+1}$ with degree $d$ bounded
by $(2n+2)^{3n}V$. Furthermore, each $D_{i}$ has degree $d_{i}$
bounded by $(3/2)^{n}(n+1)^{3n-3}V$, by Lemma \ref{degree2}.
Finally, for each $q>0$ there are at most $d_{0}^{q+1}$ choices of
$(d,d_{1},\dots,d_{q})$, which combined with Proposition 3.2 in
\cite{Hwa2} and Theorem \ref{3}, give the result.
\end{proof}

Up to constants depending only on $n$, the above bound is of the
form $V^{V^{2n+2}}$.

\subsection{Bounds on the Picard numbers}\label{Picard1}

The first result of this section is a lower bound on the Picard
number of a toroidal compactification $(X, D)$ in terms of the
number of cusps of $X\backslash D$.

\begin{theorem}\label{picard}
Let $(X,D)$ be a toroidal compactification of dimension $n\geq 2$.
Let $q$ be the number of components of $D$. Then $q<\rho(X)$, where
$\rho(X)$ is the Picard number of $X$.
\end{theorem}

\begin{proof}
Suppose by contradiction that $\rho(X)\leq q$. Let $H$ be an ample
divisor on $X$. Then there exists a non trivial relation \bdism
a_{0}H+\sum_{i=1}^{q}a_{i}D_{i}\equiv 0, \edism where $a_{i}\in\rr$.
Since $D_{i}\cdot D_{j}^{n-1}=0$ for all $j\neq i$, the
compactifying divisors are not numerically equivalent and we can
assume $a_{0}=1$. Intersecting the above relation with a curve $C$
contained in $D_{i}$, we get that $a_{i}>0$ for all $i$. On the
other hand \bdism
0=H^{n-1}\cdot\left(H+\sum_{i=1}^{q}a_{i}D_{i}\right)=H^{n}+\sum_{i=1}^{q}a_{i}D_{i}\cdot
H^{n-1}>0, \edism gives a contradiction.
\end{proof}

We now bound from above the Picard number of a toroidal
compactification in terms of its volume. Note that, because of
Corollary \ref{pairs}, there are finitely many toroidal
compactifications with given upper bound on the volume. As a result,
there is a bound on their Picard numbers depending only on an upper
bound on the volume. The next theorem shows how the techniques
developed in this paper can be nicely applied to derive an explicit bound.

\begin{theorem}\label{above}
Let $(X, D)$ be a toroidal compactification of dimension $n\geq
3$. Then
\begin{align}\notag
\rho(X)\leq 30+4(1+(n-2)(2n^{2}+3))^{2}(2n^{2}+3)^{n-2}L^{n}.
\end{align}
\end{theorem}

\begin{proof}
By Theorem \ref{Mok}, we know that $K_{X}+\alpha D$ is ample for any
$\alpha\in (\frac{1}{3}, 1)$. In particular, $2K_{X}+D$ is an ample
$\zz$-divisor. By Anghern-Siu \cite{Ang}, we know that
$K_{X}+(n^{2}+1)(2K_{X}+D)$ is ample and base point free. Let us
define
\begin{align}\notag
L_{1}=K_{X}+(n^{2}+1)(2K_{X}+D),\quad L_{2}=K_{X}+(n-2)L_{1}.
\end{align}
Note that $L_{2}$ is ample. In fact, by Theorem 1.1 it suffices to
observe that
\bdism
L_{2}=(1+(n-2)+2(n-2)(n^{2}+1))\left(K_{X}+\frac{(n-2)(n^{2}+1)}{1+(n-2)+2(n-2)(n^{2}+1)}D\right)
\edism
with
\begin{align}\notag
\frac{(n-2)(n^{2}+1)}{1+(n-2)+2(n-2)(n^{2}+1)}>\frac{1}{3}
\end{align}
for any $n\geq 3$. Since $L_{1}$ is base point free, by Bertini's
theorem, see page 137 in \cite{Griffiths}, the generic element in
the linear system $|L_{1}|$ is smooth. Let us consider $(n-2)$
distinct generic elements in $|L_{1}|$, say $S_{1}, S_{2}, ...,
S_{n-2}$, and let us denote by $S$ their intersection. $S$ is then a
smooth surface in $X$. By reiterating $(n-2)$-times the adjunction
formula, we obtain that
\begin{align}\notag
L_{2}^{2}\cdot L^{n-2}_{1}=K^{2}_{S},
\end{align}
where $K_{S}$ is ample since $L_{2}$ is ample over $X$. To compute
the intersection number $L_{2}^{2}\cdot L^{n-2}_{1}$ it is
convenient to express $L_{1}$ and $L_{2}$ in terms of the
log-canonical divisor $L=K_{X}+D$. Thus, we have
\begin{align}\notag
L_{1}=(2n^{2}+3)L-(n^{2}+2)D,
\end{align}
and
\begin{align}\notag
L_{2}=(1+(n-2)(2n^{2}+3))L-(1+(n-2)(n^{2}+2))D,
\end{align}
which then implies
\begin{align}\notag
L_{2}^{2}\cdot
L^{n-2}_{1}=(1+(n-2)(2n^{2}+3))^{2}(2n^{2}+3)^{n-2}L^{n}.
\end{align}
By the Lefschetz hyperplane theorem, see page 157 in
\cite{Griffiths}, we know that for $n=3$ then $h^{1, 1}(X)$ injects
into $h^{1, 1}(S)$, while for $n>3$ the restriction map gives an
isomorphism between $h^{1, 1}(X)$ and $h^{1, 1}(S)$. To conclude the
proof, we need an estimate on the Picard number of a surface with
ample canonical divisor in terms of $c^{2}_{1}$. Because of
Noether's inequality, see Theorem 3.1 in  \cite{Bar}, we have
\begin{align}\notag
h^{1, 1}(S)\leq 30+4K^{2}_{S},
\end{align}
so that
\begin{align}\notag
h^{1, 1}(X)\leq 30+4(1+(n-2)(2n^{2}+3))^{2}(2n^{2}+3)^{n-2}L^{n},
\end{align}
which then gives an estimate of the Picard number of $X$ in terms of
the top self intersection of $K_{X}+D$.
\end{proof}

For a two dimensional version of Theorem \ref{above} and much more
the interested reader may refer to \cite{DiCerbo1}.

\end{document}